\documentclass[10pt]{amsart}

\makeatletter
\def\blfootnote{\gdef\@thefnmark{}\@footnotetext}
\makeatother

\usepackage{epsfig}
\usepackage{graphics}
\usepackage{dcpic, pictexwd}

\usepackage[colorlinks=true,
                    linkcolor=blue,
                    urlcolor=blue,
                    citecolor=blue,
                    anchorcolor=blue]{hyperref}
\usepackage{mathtools}
\usepackage{amsmath,amssymb}
\theoremstyle{plain}

\newtheorem*{theorem*}{Theorem}
\newtheorem{theorem}{Theorem}[section]

\newtheorem{lemma}[theorem]{Lemma}
\newtheorem{proposition}[theorem]{Proposition}

\theoremstyle{remark}

\theoremstyle{Acknowledgments}

\theoremstyle{definition}

 \def\Z{{\mathbb{Z}}}

\def\mod{{\rm Mod}}

 \def\Y{{\overline{Y}}}
\def\A{{\overline{\alpha}}}
 
\begin{document}
\blfootnote{\textup{2000} \textit{Mathematics Subject Classification}:
57N05, 20F38, 20F05}
\newenvironment{prooff}{\medskip \par \noindent {\it Proof}\ }{\hfill
$\square$ \medskip \par}
    \def\sqr#1#2{{\vcenter{\hrule height.#2pt
        \hbox{\vrule width.#2pt height#1pt \kern#1pt
            \vrule width.#2pt}\hrule height.#2pt}}}
    \def\square{\mathchoice\sqr67\sqr67\sqr{2.1}6\sqr{1.5}6}
\def\pf#1{\medskip \par \noindent {\it #1.}\ }
\def\endpf{\hfill $\square$ \medskip \par}
\def\demo#1{\medskip \par \noindent {\it #1.}\ }
\def\enddemo{\medskip \par}
\def\qed{~\hfill$\square$}

 \title[Generating the Level $2$ Subgroup by Involutions] {Generating the Level $2$ Subgroup by Involutions}

\author[T{\"{u}}l\.{i}n Altun{\"{o}}z,  Naoyuki Monden,   Mehmetc\.{i}k Pamuk, and O\u{g}uz Y{\i}ld{\i}z ]{T{\"{u}}l\.{i}n Altun{\"{o}}z,  Naoyuki Monden,  Mehmetc\.{i}k Pamuk, and O\u{g}uz Y{\i}ld{\i}z}

\address{Department of Mathematics, Middle East Technical University,
 Ankara, Turkey}
\email{atulin@metu.edu.tr} 
\email{mpamuk@metu.edu.tr} \email{e171987@metu.edu.tr}
\address{Department of Mathematics, Faculty of Science, Okayama University, Okayama, Japan}
\email{n-monden@okayama-u.ac.jp}

\begin{abstract}
We obtain a minimal generating set of involutions for the level $2$ subgroup of the mapping class group of a closed nonorientable surface.
\end{abstract}

 \maketitle
  \setcounter{secnumdepth}{2}
 \setcounter{section}{0} 
 
\section{Introduction} 
Let $N_g$ be a closed nonorientable surface of genus $g\geq2$. The mapping class group $\mod(N_g)$ is defined to be the group of isotopy classes of all diffeomorphisms of $N_g$. The first homology group $H_1(N_g;\Z)$ is generated by $\lbrace x_1,x_2,\ldots, x_g \rbrace$, where $x_i$ for $1\leq i \leq g$ are the homology classes of one-sided curves as depicted in Figure~\ref{Xi}. 

 \begin{figure}[hbt!]
\begin{center}
\scalebox{0.3}{\includegraphics{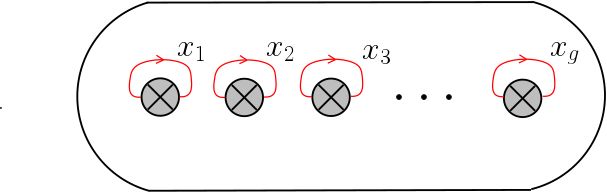}}
\caption{Generators of $H_1(N_g;\Z)$.}
\label{Xi}
\end{center}
\end{figure}

The $\Z_2$-homology classes $\overline{x}_i$ of these curves form a basis for $H_1(N_g;\Z/2\Z)$. The $\Z_2$-valued intersection pairing is a symmetric bilinear form $\langle\,,\rangle$ on $H_1(N_g;\Z/2\Z)$ satisfying $\langle \overline{x}_i, \overline{x}_j\rangle=\delta_{ij}$ for $1\leq i,j \leq g$. For more on automorphisms of $H_1(N_g;\Z/2\Z)$ and $\Z_2$-valued intersection pairings we refer the reader to~\cite{mp}.
Let $\rm{Iso}(H_1(N_g;\Z/2\Z)$ be the group of automorphisms of $H_1(N_g;\Z/2\Z)$ which preserve $\langle\,,\rangle$. \textit{The level $2$ subgroup} 
$\Gamma_2(N_g)$ of $\mod(N_g)$ is the group of isotopy classes of diffeomorphisms which act trivially on $H_1(N_g;\Z/2\Z)$. It fits into the following short exact sequence:
\[
1\longrightarrow \Gamma_2(N_g))\longrightarrow \mod(N_g) \longrightarrow \rm{Iso}(H_1(N_g;\Z/2\Z)\longrightarrow 1.
\]

For a two-sided simple closed curve $\alpha$ and a one-sided simple closed curve $\mu$ which intersect in one point, let $K$ denote a regular neighborhood of 
$\mu\cup\alpha$ that is homeomorphic to the Klein bottle with a hole. Let $M\subset K$ be a regular neighbourhood of $\mu$, which is a M{\"{o}}bius strip. 
We define the  \textit{crosscap slide} (also called \textit{$Y$-homeomorphism}) $Y_{\mu,\alpha}$ as the self-diffeomorphism of $N_g$ obtained from sliding $M$ once 
along $\alpha$ and fixing each point of the boundary of $K$ (Figure~\ref{Y}).
 
 \begin{figure}[hbt!]
\begin{center}
\scalebox{0.3}{\includegraphics{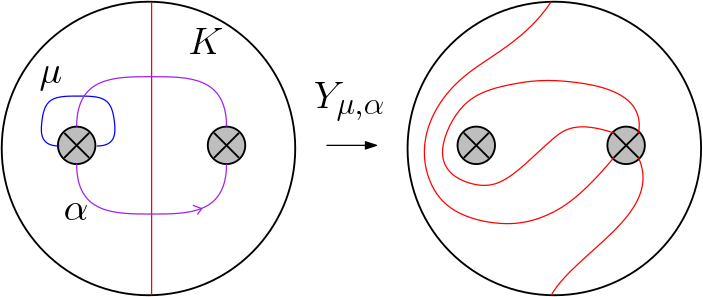}}
\caption{The crosscap slide $Y_{\mu,\alpha}$.}
\label{Y}
\end{center}
\end{figure}

For $I=\lbrace i_1,i_2,\ldots, i_k\rbrace$ a subset of $\lbrace1,2,\ldots, g\rbrace$, let $\alpha_I$ be the simple closed curve shown in Figure~\ref{A}. 
Throughout the paper, we introduce the following notations:
\begin{itemize}
\item[$\bullet$]$Y_{i_1;i_2,\ldots,i_k}=Y_{\alpha_{i_1};\alpha_{\lbrace i_1,i_2,\ldots,i_k \rbrace }}$,
\item[$\bullet$]$T_{i_1,i_2,\ldots,i_k}=T_{\alpha_{\lbrace i_1,i_2,\ldots,i_k \rbrace}}$,
\item[$\bullet$] $\alpha_{i}=\alpha_{\lbrace i,i \rbrace }$.
\end{itemize}

 \begin{figure}[hbt!]
\begin{center}
\scalebox{0.3}{\includegraphics{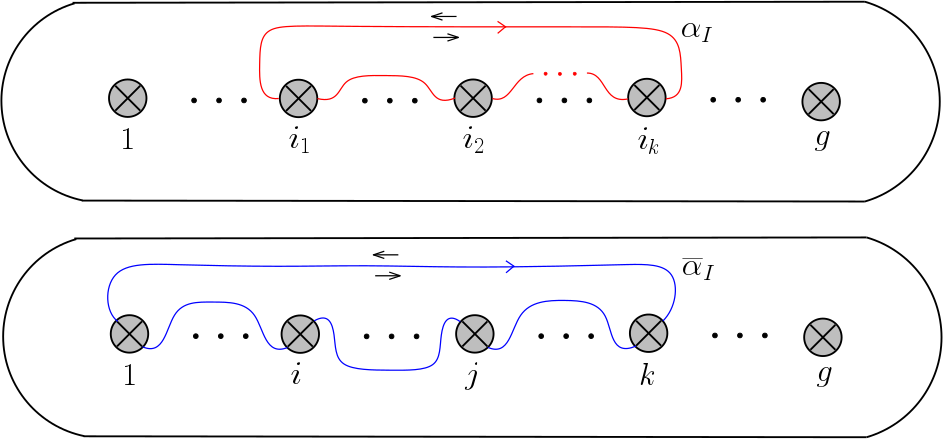}}
\caption{The curves $\alpha_I$ and  $\overline{\alpha}_I$ for $I=\lbrace i_1,i_2,\ldots,i_k \rbrace$.}
\label{A}
\end{center}
\end{figure}

Szepietowski proved that $\Gamma_2(N_g)$ is equal to the subgroup of $\mod(N_g)$ generated by all crosscap slides~\cite[Theorem $5.5$]{sz1}. 
Moreover, he proved that $\Gamma_2(N_g)$ can be generated by (infinitely many) involutions~\cite[Theorem $3.7$]{sz1}. 
In~\cite{sz3}, Szepietowski also gave a finite set of generators for $\Gamma_2(N_g)$. 
Later, Hirose and Sato reduced the number of generators of $\Gamma_2(N_g)$,  
their generating set  is as follows~\cite[Theorem $1.2$]{hs}.
\begin{theorem}\label{ts}
For $g\geq 4$, the level $2$ subgroup $\Gamma_2(N_g)$ is generated by the following two types of elements:
\begin{enumerate}
\item $Y_{i;j}$ for $i\in \lbrace 1,2,\ldots,g-1\rbrace$,  $j\in \lbrace 1,2,\ldots,g\rbrace$ and $i\neq j$;
\item $T_{1,j,k,l}^{2}$ for $1<j<k<l$.
\end{enumerate}
\end{theorem}
Note that when $g=3$, the group $\Gamma_2(N_3)$ is generated only by the elements of type $(1)$.  Hirose and Sato~\cite[Theorem $1.4$]{hs} also showed that 
for $g \geq 4$
$$
H_1(\Gamma_2(N_g); \Z) \cong (\Z/2\Z)^{{g \choose 2}+{g \choose 3}},
$$
\noindent
which in turn implies that the above generating set is minimal.

In this paper, our purpose is to give a minimal generating set of involutions for the level $2$ subgroup $\Gamma_2(N_g)$.

\section{A generating set for $\Gamma_2(N_g)$} 
Let us start this section by introducing bar notation for two-sided simple closed curves.  In the remainder of this paper, 
let  $\A_{1,i,j,k}$ and $\A_{i,j}$ be  two sided simple closed curves depicted in Figure~\ref{BA}.  Observe that when we put a bar over a two-sided simple closed curve
it passes below the in-between crosscaps.  For the ease of notation,  we also 
use the following notations:

\begin{itemize}
\item[$\bullet$]$\Y_{i,j}=Y_{\alpha_{i};\A_{i,j}}$,
\item[$\bullet$]$\overline{T}_{1,i,j,k}=T_{\A_{1,i,j,k}}$.
\end{itemize}

Recall that $\Gamma_2(N_g)$ is generated by all crosscap slides~\cite[Theorem $5.5$]{sz1}.
Let $\mathcal{Y}$ and $\overline{\mathcal{Y}}$ be the subgroups of $\Gamma_2(N_g)$ generated by elements of the form
$Y_{i,j}$ and $\overline{Y}_{i,j}$, for $i\in \lbrace 1,2,\ldots,g-1\rbrace$, $j\in \lbrace 1,2,\ldots,g\rbrace$ and $i\neq j$, respectively.
 \begin{figure}[hbt!]
\begin{center}
\scalebox{0.3}{\includegraphics{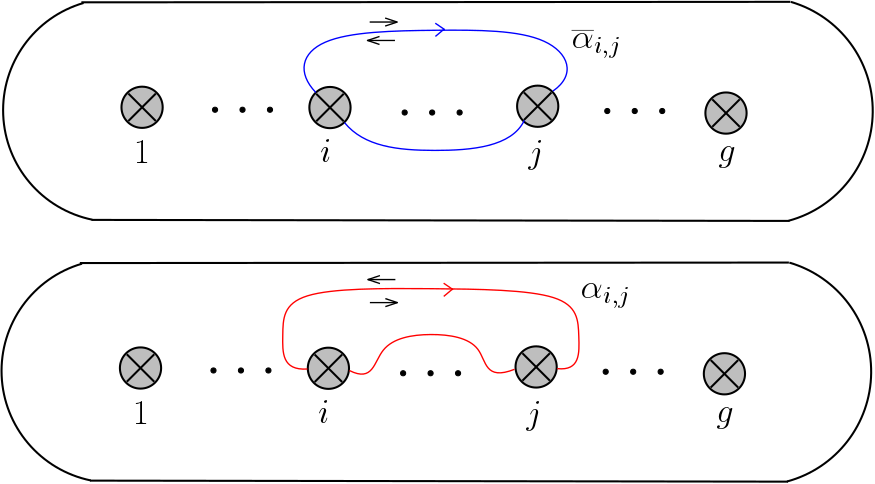}}
\caption{The curves $\A_{i,j}$ and $\alpha_{i,j}$  for $1<j<k<l$.}
\label{BA}
\end{center}
\end{figure}

\begin{lemma}\label{lemy}  
The subgroups $\mathcal{Y}$ and $\overline{\mathcal{Y}}$ are equal to each other.
\end{lemma}

\begin{proof}
Let us first  show that $\overline{\mathcal{Y}}\subseteq \mathcal{Y}$.  For  $\overline{Y}_{i, j} \in \overline{\mathcal{Y}}$,
if we assume that $\mid i-j \mid=1$, since  
\[
\overline{Y}_{i,i+1}=Y_{i,i+1} \textrm{ and } \overline{Y}_{i+1,i}=Y_{i+1,i}
\]
for all $i=1,2,\ldots,g-1$, we have $\overline{Y}_{i,j}\in \mathcal{Y}$. 
Assume now that $\mid i-j \mid>1$: 
For $i<j$, let us first consider the case $j-i=2$. It is easy to verify that
\[
\overline{Y}_{i+1,i+2}^{-1}(\alpha_i,\alpha_{i,i+2})=(\alpha_i,\overline{\alpha}_{i,i+2}),
\]
for all $i=1,\ldots, g-2$ (see Figure~\ref{Y1}).  Using $\Y_{i+1,i+2}=Y_{i+1,i+2}\in \mathcal{Y}$, we have
\[
\Y_{i,i+2}=\Y_{i+1,i+2}^{-1}Y_{i,i+2}\Y_{i+1,i+2} \in \mathcal{Y}.
\]

\begin{figure}[hbt!]
\begin{center}
\scalebox{0.35}{\includegraphics{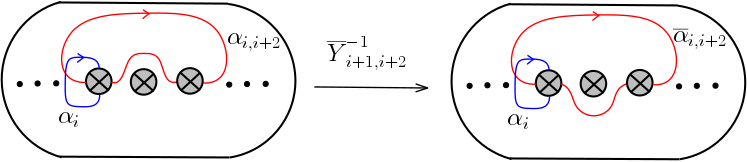}}
\caption{$\overline{Y}_{i+1,i+2}^{-1}(\alpha_i,\alpha_{i,i+2})=(\alpha_i,\overline{\alpha}_{i,i+2}).$}
\label{Y1}
\end{center}
\end{figure}

 For the case $j-i=3$, one can see that (see Figure~\ref{Y2})
\[
\Y_{i+1,i+3}^{-1}\Y_{i+2,i+3}^{-1}(\alpha_i,\alpha_{i,i+3})=(\alpha_i,\overline{\alpha}_{i,i+3}).
\]
\begin{figure}[hbt!]
\begin{center}
\scalebox{0.35}{\includegraphics{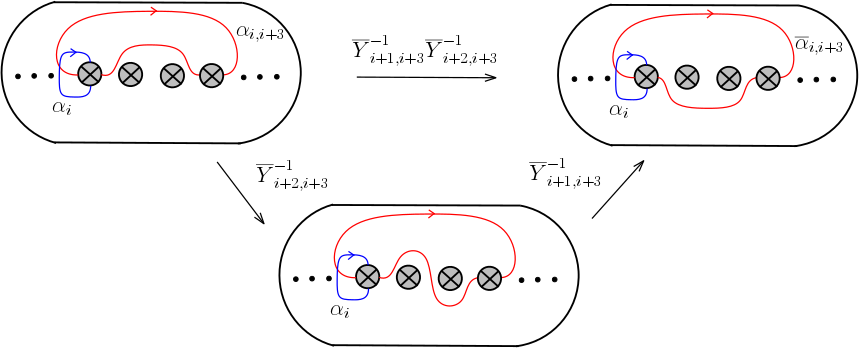}}
\caption{$\Y_{i+1,i+3}^{-1}\Y_{i+2,i+3}^{-1}(\alpha_i,\alpha_{i,i+3})=(\alpha_i,\overline{\alpha}_{i,i+3}).$}
\label{Y2}
\end{center}
\end{figure}
Now, since $\Y_{i+2,i+3}$ and $\Y_{i+1,i+3}$ are all contained in $\mathcal{Y}$  we have
\[
\displaystyle
\Y_{i,i+3}=(\Y_{i+1,i+3}^{-1}\Y_{i+2,i+3}^{-1})Y_{i,i+2}(\Y_{i+1,i+3}^{-1}\Y_{i+2,i+3}^{-1})^{-1}\in \mathcal{Y} 
\]
for $i=1,\ldots, g-3$.  For the remaining $i<j$ cases, one can see that 
\[
\displaystyle
\Y_{i,j}=(\Y_{i+1,j}^{-1}\Y_{i+2,j}^{-1}\cdots \Y_{j-1,j}^{-1})Y_{i,j}(\Y_{i+1,j}^{-1}\Y_{i+2,j}^{-1}\cdots \Y_{j-1,j}^{-1})^{-1}\in \mathcal{Y}
\]
for all $i=1,\ldots, g-1$, $j=1,\ldots,g$.

Now, we consider the cases where $i>j$. For $i-j>2$, we have (see Figure~\ref{Y3})
\[
\Y_{i+1,i}(\alpha_{i+2},\alpha_{i,i+2})=(\alpha_{i+2},\overline{\alpha}_{i,i+2}).
\]
\begin{figure}[hbt!]
\begin{center}
\scalebox{0.35}{\includegraphics{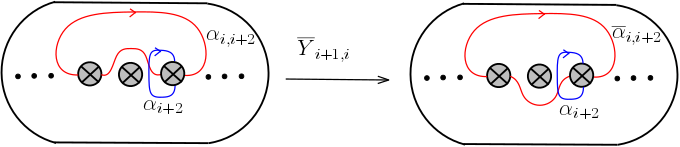}}
\caption{$\Y_{i+1,i}(\alpha_{i+2},\alpha_{i,i+2})=(\alpha_{i+2},\overline{\alpha}_{i,i+2}).$}
\label{Y3}
\end{center}
\end{figure}

Using $\Y_{i+1, i}\in \mathcal{Y}$ for $i=1,\ldots,g-3$, we get
\[
\Y_{i+2, i}=\Y_{i+1, i}Y_{i+2,i}\Y_{i+1, i}^{-1}\in \mathcal{Y}.
\]
As before, for all $i=1,\ldots, g-1$ and $j=1,\ldots, g-2$, we have
\[
\Y_{i,j}=(\Y_{i,j-1}\cdots\Y_{i,i+1})Y_{i,j}(\Y_{i,j-1}\cdots\Y_{i,i+1})^{-1}\in \mathcal{Y}.
\]
Thus, $\Y_{i,j}\in \mathcal{Y}$ for $1\leq i<j \leq g$.  
Since we cover all the cases, we have shown  $\overline{\mathcal{Y}}\subseteq \mathcal{Y}$.
For the reverse inclusion, note that we have the following equalities
\begin{equation}
    Y_{i,j}=
    \begin{cases}
     (\Y_{i+1,j}^{-1}\cdots \Y_{j-1,j}^{-1})^{-1}\Y_{i,j}  (\Y_{i+1,j}^{-1}\cdots \Y_{j-1,j}^{-1}), & \text{if}\ i<j, \\
       (\Y_{i,j-1}^{-1}\cdots \Y_{i,i+1}^{-1})^{-1}\Y_{i,j}(\Y_{i,j-1}^{-1}\cdots \Y_{i,i+1}^{-1}), & \text{if}\ i>j,
    \end{cases}
  \end{equation}
which immediately imply that $\mathcal{Y}\subseteq \overline{\mathcal{Y}}$.
\end{proof}

Next, we present  a minimal generating set for the level $2$ subgroup $\Gamma_2(N_g)$ (cf.~\cite[Theorem $1.2$]{hs}).
\begin{theorem}\label{t1}
For $g\geq4$, the level $2$ subgroup $\Gamma_2(N_g)$ can be generated by 
\begin{enumerate}
\item $\Y_{i,j}$ for $i \in \lbrace 1,\ldots,g-1 \rbrace$, $j\in \lbrace 1,\ldots,g \rbrace$ and $i\neq j$,
\item $\overline{T}_{1,i,j,k}^{2}$ for $1<i<j<k$.
\end{enumerate}
\end{theorem}
\begin{proof}
Let $G$ be the subgroup of $\Gamma_2(N_g)$ generated by the elements given in $(1)$ and $(2)$. 
Since by Lemma~\ref{lemy} we have $\mathcal{Y}=\overline{\mathcal{Y}}$, it is enough to prove that ${T}_{1,i,j,k}^{2}$ 
is contained in the subgroup $G$ for $1<i<j<k$.

It is easy to check that 
\[
\Y_{i+1,j}^{-1}\cdots \Y_{j-2,j}^{-1}\Y_{j-1,j}^{-1}(\alpha_{1,i,j,k})=\overline{\alpha}_{1,i,j,k}.
\]
Thus
\[
\overline{T}_{1,i,j,k}^{2}=(\Y_{i+1,j}^{-1}\cdots \Y_{j-2,j}^{-1}\Y_{j-1,j}^{-1})T_{1,i,j,k}^{2}(\Y_{i+1,j}^{-1}\cdots \Y_{j-2,j}^{-1}\Y_{j-1,j}^{-1})^{-1},
\]
which implies that 
\[
T_{1,i,j,k}^{2}=(\Y_{i+1,j}^{-1}\cdots \Y_{j-2,j}^{-1}\Y_{j-1,j}^{-1})^{-1}\overline{T}_{1,i,j,k}^{2}(\Y_{i+1,j}^{-1}\cdots \Y_{j-2,j}^{-1}\Y_{j-1,j}^{-1})\in G
\]
for $1<i<j<k$. This completes the proof.
\end{proof}
\section{Involution generators for $\Gamma_2(N_g)$} 
In this section, we give a generating set of involutions for $\Gamma_2(N_g)$.
Throughout this section, consider the surface $N_g$ as shown in Figure~\ref{R1} so that it is invariant under the reflection $R$ about the indicated plane. 
Note that, $R$ acts trivially on $H_1(N_g;\Z/2\Z)$, which implies 
that it is an element of the subgroup $\Gamma_2(N_g)$.

\begin{figure}[hbt!]
\begin{center}
\scalebox{0.2}{\includegraphics{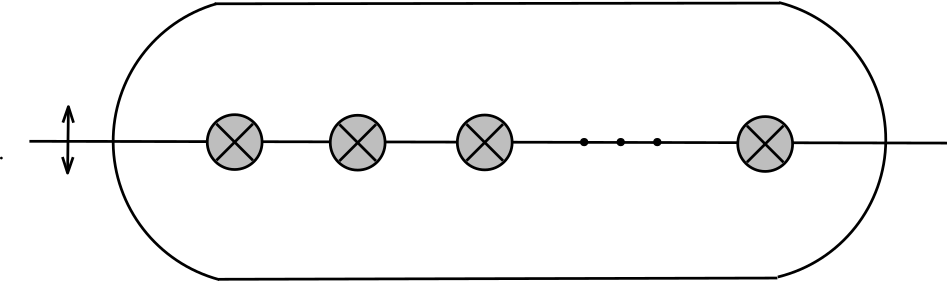}}
\caption{The reflection $R$.}
\label{R1}
\end{center}
\end{figure}

\begin{proposition}\label{p1}
For $g\geq4$, the group $\Gamma_2(N_g)$ can be generated by 
\begin{enumerate}
\item $R$,
\item $R\Y_{i,j}$ for $i \in \lbrace 1,\ldots,g-1 \rbrace$, $j\in \lbrace 1,\ldots,g \rbrace$ and $i\neq j$,
\item $R\Y_{1,i}Y_{\alpha_{k},\overline{\alpha}_{j,k}}\overline{T}_{1,i,j,k}^{2}$ for $1<i<j<k$.
\end{enumerate}
\end{proposition}
\begin{proof}
Let $G$ be the subgroup generated by the elements listed in the statement of the proposition.  
Since the subgroup $G$ contains $R$ and $R\Y_{i,j}$, it also contains 
\[
\Y_{i,j}=R(R\Y_{i,j})
\]
for $i \in \lbrace 1,\ldots,g-1 \rbrace$, $j\in \lbrace 1,\ldots,g \rbrace$ and $i\neq j$.   Recall that  $\overline{\mathcal{Y}}$ is generated by such elements, hence 
$\overline{\mathcal{Y}}\subseteq G$. By Theorem~\ref{t1}, it remains to prove that $\overline{T}_{1, i, j, k}^{2}$ also belongs to $G$.  Now, it is easy to see that $G$ contains 
\[
\Y_{1,i}Y_{\alpha_{k},\overline{\alpha}_{j,k}}\overline{T}_{1,i,j,k}^{2}=R(R\Y_{1,i}Y_{\alpha_{k},\overline{\alpha}_{j,k}}\overline{T}_{1,i,j,k}^{2}).
\]
The elements $Y_{\alpha_{k},\overline{\alpha}_{j,k}}$ are contained in $\mathcal{Y}= \overline{\mathcal{Y}}$  by \cite[Lemma $3.5$]{sz3} and Lemma~\ref{lemy}. 
Since the elements $\Y_{1,i}$ are also contained in $G$, one can conclude that $\overline{T}_{1,i,j,k}^{2}\in G$ for $1<i<j<k$, which finishes the proof.
\end{proof}
\begin{lemma}\label{lem2}
The reflection $R$ can be expressed as a product of finitely $Y$-homeomorphisms.  In particular  
$$R=\Y_{g-1,g}\Y_{g-2,g}\cdots\Y_{1,g}.$$
\end{lemma}
\begin{proof}
It follows from the proof of~\cite[Lemma $3.4$]{sz1} that $R$ can be written as
\begin{eqnarray*}
R&=&Y_{g-1,g}T_{\alpha_{g-1,g}}^{-1}Y_{g-2,g-1}T_{\alpha_{g-1,g}} (T_{\alpha_{i+1,i+2}}T_{\alpha_{i+2,i+3}}\cdots T_{\alpha_{g-1,g}})^{-1}\\
&&Y_{i,i+1} (T_{\alpha_{i+1,i+2}}T_{\alpha_{i+2,i+3}}\cdots T_{\alpha_{g-1,g}})\cdots
 (T_{\alpha_{2,3}}\cdots T_{\alpha_{g-1,g}})^{-1}Y_{1,2}(T_{\alpha_{2,3}}\cdots T_{\alpha_{g-1,g}}).
\end{eqnarray*}
It is easy to see that 
\[
(T_{\alpha_{i+1,i+2}}T_{\alpha_{i+2,i+3}}\cdots  T_{\alpha_{g-1,g}})^{-1}(\alpha_i,\alpha_{i,i+1})=(\alpha_i,\A_{1,g}),
\]
from which we obtain
\[
\Y_{i,g}=(T_{\alpha_{i+1,i+2}}T_{\alpha_{i+2,i+3}}\cdots  T_{\alpha_{g-1,g}})^{-1}Y_{i,i+1}(T_{\alpha_{i+1,i+2}}T_{\alpha_{i+2,i+3}}\cdots  T_{\alpha_{g-1,g}}),
\]
for $i\in \lbrace 1,\ldots ,g-1\rbrace$. This completes the proof.
\end{proof}
 Next, we show that the elements mentioned in  Theorem~\ref{p1} are all involutions.  We already know that the reflection $R$ is an involution.

\begin{lemma}\label{lem3}
If $g\geq 4$, then the elements $R\Y_{i,j}^{\pm 1}$ are all  involutions for $i \in \lbrace 1,\ldots,g-1 \rbrace$, $j\in \lbrace 1,\ldots,g \rbrace$ and $i\neq j$.
\end{lemma}
\begin{proof}
It is enough to see that $R(\alpha_i,\overline{\alpha}_{i,j})=(\alpha_i^{-1},\overline{\alpha}_{i,j}^{-1})$.
\end{proof}
\begin{lemma}\label{lem4}
If $g\geq 4$, then the elements $R\Y_{1,i}Y_{\alpha_{k},\overline{\alpha}_{j,k}}\overline{T}_{1,i,j,k}^{2}$ are all involutions for $1<i<j<k$.
\end{lemma}
\begin{proof}
First of all, it is easy verify that 
\[
R(\A_{1,i},\A_{j,k})=(\A_{1,i}^{-1},\A_{j,k}^{-1}) \ \textrm{and} \ 
R(\alpha_i,\alpha_k)=(\alpha_i^{-1},\alpha_k^{-1}).
\]
Then we have the following:
\begin{eqnarray*}
R\Y_{1,i}Y_{\alpha_k,\A_{j,k}}^{-1}R^{-1}&=&\Y_{1,i}^{-1}Y_{\alpha_k,\A_{j,k}}\\
&=&Y_{\alpha_k,\A_{j,k}}\Y_{1,i}^{-1},
\end{eqnarray*}
where the last identity follows from the commutativity of crosscap slides $\Y_{1,i}$ and $Y_{\alpha_k,\A_{j,k}}$. 
Observe that, this implies $R\Y_{1,i}Y_{\alpha_k,\A_{j,k}}^{-1}$ is an involution. Moreover, since
\[
R\Y_{1,i}Y_{\alpha_k,\A_{j,k}}^{-1}(\A_{1,i,j,k})=\A_{1,i,j,k}^{-1}
\]
it follows that $R\Y_{1,i}Y_{\alpha_{k},\overline{\alpha}_{j,k}}\overline{T}_{1,i,j,k}^{2}$ is also an involution.
\end{proof}

Finally, we present our involution generators.  Note that in the following, the number of involution generators is 
equal to ${g \choose 2}+{g \choose 3}$ which is the minimal possible number of generators for $\Gamma_{2}(N_g)$.

\begin{theorem}\label{t2}
For $g\geq 5$ and odd, $\Gamma_{2}(N_g)$ is generated by the following involutions:
\begin{enumerate}
\item $R\Y_{1,g}, R\Y_{2,g}^{-1},\ldots,R\Y_{g-2,g}, R\Y_{g-1,g}^{-1}$,
\item $R\Y_{i,j}$ for $i,j \in\lbrace 1,2,\ldots,g-1\rbrace$ and $i\neq j$,
\item $R\Y_{1,i}Y_{\alpha_k,\A_{j,k}}^{-1}\overline{T}_{1,i,j,k}^{2}$ for $1<i<j<k$.
\end{enumerate}
For $g\geq 4$ and even, $\Gamma_{2}(N_g)$ is generated by the following involutions:
\begin{enumerate}
\item $R,R\Y_{1,g}, R\Y_{2,g}^{-1},\ldots,R\Y_{g-2,g}$,
\item $R\Y_{i,j}$ for $i,j \in\lbrace 1,2,\ldots,g-1\rbrace$ and $i\neq j$,
\item $R\Y_{1,i}Y_{\alpha_k,\A_{j,k}}^{-1}\overline{T}_{1,i,j,k}^{2}$ for $1<i<j<k$.
\end{enumerate}
\end{theorem}
\begin{proof}
Let $G$ denote the subgroup of $\Gamma_2(N_g)$ generated by the elements listed in Theorem~\ref{t2}. It follows from lemmata~\ref{lem3} and~\ref{lem4} that the generators of the group $G$ are involutions.

Let us first assume that $g\geq5$ and odd. By Proposition~\ref{p1}, it is enough to prove that $R$ is contained in the subgroup $G$. It follows from Lemma~\ref{lem2} the reflection $R$ can be expresses as 
\begin{eqnarray*}
R&=&\Y_{g-1,g}\Y_{g-2,g}\cdots\Y_{1,g}\\
&=&R^2\Y_{g-1,g}\Y_{g-2,g}R^2\Y_{g-3,g}^{-1} \cdots R^2\Y_{2,g}\Y_{1,g}\\
&=&R\Y_{g-1,g}^{-1}R\Y_{g-2,g}R\Y_{g-3,g}^{-1}R\cdots R\Y_{2,g}^{-1}R\Y_{1,g},
\end{eqnarray*}
which is contained in the subgroup $G$ using $R\Y_{i,g}^{-1}R=\Y_{i,g}$.

Assume now that $g\geq4$ and even. In this case, by Proposition~\ref{p1}, it suffices to show that the subgroup $G$ contains the element $\Y_{g-1,g}$. The following element is contained in the subgroup $G$:
\begin{eqnarray*}
R(R\Y_{g-2,g}R\Y_{g-3,g}^{-1}R\Y_{g-4,g}R\cdots R\Y_{2,g}^{-1}R\Y_{1,g})\\
=R(R\Y_{g-2,g}R^2\Y_{g-3,g}\Y_{g-4,g}\cdots R^2\Y_{2,g}\Y_{1,g})\\
=\Y_{g-2,g}\Y_{g-3,g}\Y_{g-4,g}\cdots \Y_{2,g}\Y_{1,g},
\end{eqnarray*}
using again $R\Y_{i,g}^{-1}R=\Y_{i,g}$. One can conclude that $\Y_{g-1,g}\in G$ since $R\in G$ by Lemma~\ref{lem2}, which finishes the proof.
\end{proof}

\end{document}